\mag=\magstep1 
\input epsf 
\documentclass[10pt,a4paper]{amsart} 
\usepackage{amssymb,latexsym} 

\def\Cal{\mathcal}

\def\<<{\langle } 
\def\>>{\rangle }

\numberwithin{equation}{section} 
\newtheorem{theorem}{Theorem}[section] 
\newtheorem{proposition}[theorem]{Proposition} 
\newtheorem{proposition-definition}[theorem]{Proposition-Definition} 
 
\newtheorem{definition}[theorem]{Definition} 
 
\newtheorem{remark}[theorem]{Remark}

\newtheorem{problem}[theorem]{Problem}

\def\sgn{\operatorname{sgn}}

\def\<{\langle} 
\def\>{\rangle} 

\def\Rp{{\scriptstyle \mathbf{R}_+}}
\def\Rn{{\scriptstyle -\mathbf{R}_+}}
\def\Ri{{\scriptstyle \sqrt{-1}\cdot\mathbf{R}_+}}
\def\Rj{{\scriptstyle -\sqrt{-1}\cdot\mathbf{R}_+}}
\def\comment#1{} 
 
\begin{document} 
 
\title{Arithmetic-geometric means for hyperelliptic curves and 
Calabi-Yau varieties} 
\dedicatory{Dedicated to Professor~Kazuo~Okamoto on 
the~occasion of his~60-th~birthday.}

\author{Keiji Matsumoto and Tomohide Terasoma}
 
\begin{abstract}
In this paper, we define a generalized 
arithmetic-geometric mean $\mu_g$  
among $2^g$ terms  
motivated by $2\tau$-formulas 
of theta constants. 
By using Thomae's formula,  
we give two expressions of $\mu_g$
when initial terms satisfy some conditions.
One is given in terms of period integrals of
a hyperelliptic curve $C$ of 
genus $g$.
The other is by a period integral of 
a certain Calabi-Yau $g$-fold given as 
a double cover of the $g$-dimensional projective space $\mathbf{P}^g$.

\end{abstract}


\maketitle 

\makeatletter 
\renewcommand{\@evenhead}{\tiny \thepage \hfill 
Arithmetic-Geometric Means
\hfill}

\renewcommand{\@oddhead}{\tiny \hfill 
K.Matsumoto and T.Terasoma
\hfill \thepage} 

\section{Introduction}
Let $\{a_{n,0}\}_n$ and $\{a_{n,1}\}_n$ be positive real sequences 
defined by  the recurrence relations 
\begin{equation}
\label{classical recursive}
a_{n+1,0}=\frac{a_{n,0}+a_{n,1}}{2},\quad a_{n+1,1}=\sqrt{a_{n,0}a_{n,1}}
,
\end{equation}
and initial terms $a_{0,0}=a_0$, $a_{0,1}=a_1$ with $0<a_1<a_0$. 
One can easily show that $\{a_{n,0}\}_n$ and $\{a_{n,1}\}_n$ 
have a common limit, which is called 
the arithmetic-geometric mean of $a_0$ and $a_1$, and is denoted by 
$\mu_1(a_0,a_1)$. By the homogeneity of the arithmetic and geometric means, 
we have 
$\mu_1(ca_0,ca_1)=c\mu_1(a_0,a_1)$ for any positive real number $c$.

On the other hand,  two Jacobi's theta constants 
$\theta_0$ and $\theta_1$ satisfy the following $2\tau$-formulas:
\begin{equation}
\label{classical 2 isog}
\theta_0(2\tau)^2=\frac{\theta_0(\tau)^2+\theta_1(\tau)^2}{2},
\quad 
\theta_1(2\tau)^2=\theta_0(\tau)\theta_1(\tau),
\end{equation}
where 
$$
\theta_{i}(\tau)=\sum_{n\in \mathbf{Z}}\exp(\pi\sqrt{-1} (n^2\tau+i n))
,\qquad i=0,1,
$$
and $\tau$ belongs to the upper half space $\mathbf{H}$.
If we find an element $\tau\in \mathbf{H}$ such that 
$\theta_1(\tau)^2/\theta_0(\tau)^2=a_1/a_0$
for given initial terms $a_0$ and $a_1$, 
then we have
$$\frac{a_0}{\mu_1(a_0,a_1)}
=\frac{\theta_0(\tau)^2}{\mu_1(\theta_0(\tau)^2,\theta_1(\tau)^2)}=
\frac{\theta_0(\tau)^2}{\mu_1(\theta_0(2^n\tau)^2,\theta_1(2^n\tau)^2)}
=
\theta_0(\tau)^2$$
by (1.1), (1.2) and  $\lim\limits_{n\to\infty}\theta_i(2^n\tau)=1.$
Moreover, Jacobi's formula between $\theta_0(\tau)^2$ and 
an elliptic integral
implies that 
$$\frac{a_0}{\mu_1(a_0,a_1)}=
\frac{2}{\pi}\int_0^1\frac{dx}{\sqrt{(1-x^2)(1-k^2x^2)}},\quad 
k=\frac{\sqrt{a_0^2-a_1^2}}{a_0}.$$

In this paper, we define a generalized 
arithmetic-geometric mean $\mu_g$  
among $2^g$ terms $(\dots,a_{I},\dots)$ $(I\in \mathbf{F}_2^g)$ 
motivated by the $2\tau$-formulas 
(\ref{twice formular})
of theta constants obtained by Theorem 2 in \cite{I} p.139. 
By using Thomae's formula,  
we give two expressions of $\mu_g$
whose initial terms are given as (\ref{initial}) for some 
$2g+1$ real numbers $p_j$.
One is given in terms of period integrals of the 
hyperelliptic curve $C$ of 
genus $g$ represented by the double cover of the complex projective line 
$\mathbf{P}^1$ branching at $\infty$ and $2g+1$ points $p_j.$ 
The other is by a period integral of the Calabi-Yau $g$-fold which is 
the double cover of the $g$-dimensional projective space $\mathbf{P}^g$ 
branching 
along the dual hyperplanes of
the images of 
$\infty$ and $p_j$ $(j=1,\dots,2g+1)$ under the Veronese 
embedding of $\mathbf P^1$ into $\mathbf P^g$.

In 1876, Borchardt studied in \cite{B} the case of $g=2$: 
the generalized arithmetic-geometric mean $\mu_2$ of 
$a=(a_{00},a_{01},a_{10},a_{11})$ was given by 
the iteration of four means
\begin{eqnarray*}
& &\frac{a_{00}+a_{01}+a_{10}+a_{11}}{ 4},\quad 
\frac{\sqrt{a_{00}a_{01}}+\sqrt{a_{11}a_{10}}}{ 2},\\
& &\frac{\sqrt{a_{00}a_{10}}+\sqrt{a_{11}a_{01}}}{ 2},\qquad 
\frac{\sqrt{a_{00}a_{11}}+\sqrt{a_{10}a_{01}}}{ 2},
\end{eqnarray*}
and $\mu_2(a)$ was expressed in terms of 
period integrals of a hyperelliptic curve of genus $2$. 
Mestre showed in \cite{Me} that 
$\mu_2(a)$ could be expressed 
in terms of $\mu_1$ and some algebraic functions of $a$
when 
$$a_{00}>a_{01}=a_{10}>a_{11},\quad a_{00}a_{11}>a_{01}a_{10}.$$

\section{Comparison to theta constants}
\label{comparison to theta}
We define a hyperelliptic curve
$C$ of genus $g$
by
$$
C:y^2=(x-p_1)\cdots (x-p_{2g+1}),
$$
where $p_j$'s are real numbers satisfying
$p_1< \dots <p_{2g+1}$.
As in \cite{M2} p.76,
we choose the cycles $A_1,\dots, A_g,B_1,\dots, B_g$ in the union of the
following two sheets (I),(II) in Figure \ref{cycles}. 
Here $\mathbf R_+$ is the set of non-negative
real numbers, the range of values of $y$ is written, and
the cycles in the sheet II are written in thick lines.
\begin{figure}[hbt]

\unitlength 1pt
\begin{picture}(302.3105,219.7975)( -3.5566,-242.2041)
\put(213.3957,-56.9055){\makebox(0,0){$\cdots$}}%
\put(28.4528,-56.9055){\makebox(0,0)[rb]{$p_1$}}%
\put(113.8110,-56.9055){\makebox(0,0)[rb]{$p_3$}}%
\put(270.3012,-56.9055){\makebox(0,0)[rb]{$p_{2g+1}$}}%
\put(71.1319,-56.9055){\makebox(0,0)[lb]{$p_2$}}%
\put(156.4902,-56.9055){\makebox(0,0)[lb]{$p_4$}}%
\put(156.4902,-28.4528){\makebox(0,0){Sheet I (solid line)}}%
\put(49.7923,-49.7923){\makebox(0,0){$\Rn$}}%
\put(49.7923,-64.0187){\makebox(0,0){$\Rp$}}%
\put(135.1506,-49.7923){\makebox(0,0){$\Rp$}}%
\put(135.1506,-64.0187){\makebox(0,0){$\Rn$}}%
\put(92.4715,-64.0187){\makebox(0,0){$\Ri$}}%
\put(177.8297,-64.0187){\makebox(0,0){$\Rj$}}%
\put(213.3957,-113.8110){\makebox(0,0){$\cdots$}}%
\put(28.4528,-113.8110){\makebox(0,0)[rb]{$p_1$}}%
\put(113.8110,-113.8110){\makebox(0,0)[rb]{$p_3$}}%
\put(270.3012,-113.8110){\makebox(0,0)[rb]{$p_{2g+1}$}}%
\put(71.1319,-113.8110){\makebox(0,0)[lb]{$p_2$}}%
\put(156.4902,-113.8110){\makebox(0,0)[lb]{$p_4$}}%
\put(156.4902,-85.3583){\makebox(0,0){Sheet II (dotted line)}}%
\put(49.7923,-106.6978){\makebox(0,0){$\Rp$}}%
\put(49.7923,-120.9242){\makebox(0,0){$\Rn$}}%
\put(135.1506,-106.6978){\makebox(0,0){$\Rn$}}%
\put(135.1506,-120.9242){\makebox(0,0){$\Rp$}}%
\put(92.4715,-120.9242){\makebox(0,0){$\Rj$}}%
\put(177.8297,-120.9242){\makebox(0,0){$\Ri$}}%
%
\special{pn 13}%
\special{pa 394 788}%
\special{pa 985 788}%
\special{fp}%
\special{pa 1575 788}%
\special{pa 2166 788}%
\special{fp}%
\special{pa 3544 788}%
\special{pa 4134 788}%
\special{fp}%
\put(298.7539,-56.9055){\makebox(0,0)[lb]{$\infty$}}%
%
\special{pn 13}%
\special{pa 394 1575}%
\special{pa 985 1575}%
\special{fp}%
\special{pa 1575 1575}%
\special{pa 2166 1575}%
\special{fp}%
\special{pa 3544 1575}%
\special{pa 4134 1575}%
\special{fp}%
\put(298.7539,-113.8110){\makebox(0,0)[lb]{$\infty$}}%
\put(213.3957,-199.1693){\makebox(0,0){$\cdots$}}%
\put(28.4528,-199.1693){\makebox(0,0)[rb]{$p_1$}}%
\put(113.8110,-199.1693){\makebox(0,0)[rb]{$p_3$}}%
\put(270.3012,-199.1693){\makebox(0,0)[rb]{$p_{2g+1}$}}%
\put(71.1319,-199.1693){\makebox(0,0)[lb]{$p_2$}}%
\put(156.4902,-199.1693){\makebox(0,0)[lb]{$p_4$}}%
\put(156.4902,-142.2638){\makebox(0,0){Cycles}}%
%
\special{pn 13}%
\special{pa 394 2756}%
\special{pa 985 2756}%
\special{fp}%
\special{pa 1575 2756}%
\special{pa 2166 2756}%
\special{fp}%
\special{pa 3544 2756}%
\special{pa 4134 2756}%
\special{fp}%
\put(298.7539,-199.1693){\makebox(0,0)[lb]{$\infty$}}%
%
\special{pn 8}%
\special{pa 197 2560}%
\special{pa 985 2560}%
\special{fp}%
\special{sh 1}%
\special{pa 985 2560}%
\special{pa 919 2540}%
\special{pa 933 2560}%
\special{pa 919 2579}%
\special{pa 985 2560}%
\special{fp}%
%
\special{pn 8}%
\special{pa 985 2560}%
\special{pa 1182 2560}%
\special{fp}%
\special{pa 1182 2560}%
\special{pa 1182 2953}%
\special{fp}%
%
\special{pn 8}%
\special{pa 1182 2953}%
\special{pa 394 2953}%
\special{fp}%
\special{sh 1}%
\special{pa 394 2953}%
\special{pa 460 2973}%
\special{pa 446 2953}%
\special{pa 460 2934}%
\special{pa 394 2953}%
\special{fp}%
%
\special{pn 8}%
\special{pa 394 2953}%
\special{pa 197 2953}%
\special{fp}%
\special{pa 197 2953}%
\special{pa 197 2560}%
\special{fp}%
\put(71.1319,-184.9429){\makebox(0,0)[lb]{$A_1$}}%
%
\special{pn 8}%
\special{pa 1378 2560}%
\special{pa 2166 2560}%
\special{fp}%
\special{sh 1}%
\special{pa 2166 2560}%
\special{pa 2100 2540}%
\special{pa 2114 2560}%
\special{pa 2100 2579}%
\special{pa 2166 2560}%
\special{fp}%
%
\special{pn 8}%
\special{pa 2166 2560}%
\special{pa 2363 2560}%
\special{fp}%
\special{pa 2363 2560}%
\special{pa 2363 2953}%
\special{fp}%
%
\special{pn 8}%
\special{pa 1575 2953}%
\special{pa 1378 2953}%
\special{fp}%
\special{pa 1378 2953}%
\special{pa 1378 2560}%
\special{fp}%
\put(156.4902,-184.9429){\makebox(0,0)[lb]{$A_2$}}%
%
\special{pn 8}%
\special{pa 2363 2953}%
\special{pa 1575 2953}%
\special{fp}%
\special{sh 1}%
\special{pa 1575 2953}%
\special{pa 1641 2973}%
\special{pa 1627 2953}%
\special{pa 1641 2934}%
\special{pa 1575 2953}%
\special{fp}%
%
\special{pn 8}%
\special{pa 1871 2756}%
\special{pa 1871 2363}%
\special{fp}%
\special{sh 1}%
\special{pa 1871 2363}%
\special{pa 1851 2429}%
\special{pa 1871 2415}%
\special{pa 1890 2429}%
\special{pa 1871 2363}%
\special{fp}%
%
\special{pn 8}%
\special{pa 1871 2363}%
\special{pa 3052 2363}%
\special{fp}%
\special{sh 1}%
\special{pa 3052 2363}%
\special{pa 2986 2343}%
\special{pa 3000 2363}%
\special{pa 2986 2382}%
\special{pa 3052 2363}%
\special{fp}%
%
\special{pn 8}%
\special{pa 2658 3150}%
\special{pa 1871 3150}%
\special{dt 0.045}%
\special{pa 1871 3150}%
\special{pa 1871 2756}%
\special{dt 0.045}%
%
\special{pn 8}%
\special{pa 3052 2363}%
\special{pa 3741 2363}%
\special{fp}%
%
\special{pn 8}%
\special{pa 3741 2363}%
\special{pa 3741 2560}%
\special{fp}%
\special{sh 1}%
\special{pa 3741 2560}%
\special{pa 3760 2494}%
\special{pa 3741 2507}%
\special{pa 3721 2494}%
\special{pa 3741 2560}%
\special{fp}%
%
\special{pn 8}%
\special{pa 3741 2560}%
\special{pa 3741 2756}%
\special{fp}%
%
\special{pn 8}%
\special{pa 3741 2756}%
\special{pa 3741 3150}%
\special{dt 0.045}%
%
\special{pn 8}%
\special{pa 3741 3150}%
\special{pa 2658 3150}%
\special{dt 0.045}%
\special{sh 1}%
\special{pa 2658 3150}%
\special{pa 2724 3170}%
\special{pa 2710 3150}%
\special{pa 2724 3130}%
\special{pa 2658 3150}%
\special{fp}%
%
\special{pn 8}%
\special{pa 689 2756}%
\special{pa 689 2363}%
\special{fp}%
\special{sh 1}%
\special{pa 689 2363}%
\special{pa 670 2429}%
\special{pa 689 2415}%
\special{pa 709 2429}%
\special{pa 689 2363}%
\special{fp}%
%
\special{pn 8}%
\special{pa 689 2363}%
\special{pa 689 2166}%
\special{fp}%
%
\special{pn 8}%
\special{pa 689 2166}%
\special{pa 2658 2166}%
\special{fp}%
\special{sh 1}%
\special{pa 2658 2166}%
\special{pa 2592 2146}%
\special{pa 2606 2166}%
\special{pa 2592 2186}%
\special{pa 2658 2166}%
\special{fp}%
%
\special{pn 8}%
\special{pa 2658 2166}%
\special{pa 3938 2166}%
\special{fp}%
%
\special{pn 8}%
\special{pa 3938 2166}%
\special{pa 3938 2560}%
\special{fp}%
\special{sh 1}%
\special{pa 3938 2560}%
\special{pa 3957 2494}%
\special{pa 3938 2507}%
\special{pa 3918 2494}%
\special{pa 3938 2560}%
\special{fp}%
%
\special{pn 8}%
\special{pa 3938 2560}%
\special{pa 3938 2756}%
\special{fp}%
%
\special{pn 8}%
\special{pa 3938 2756}%
\special{pa 3938 3347}%
\special{dt 0.045}%
%
\special{pn 8}%
\special{pa 3938 3347}%
\special{pa 2166 3347}%
\special{dt 0.045}%
\special{sh 1}%
\special{pa 2166 3347}%
\special{pa 2232 3367}%
\special{pa 2218 3347}%
\special{pa 2232 3327}%
\special{pa 2166 3347}%
\special{fp}%
%
\special{pn 8}%
\special{pa 2166 3347}%
\special{pa 689 3347}%
\special{dt 0.045}%
\special{pa 689 3347}%
\special{pa 689 2756}%
\special{dt 0.045}%
\put(49.7923,-170.7165){\makebox(0,0)[rb]{$B_1$}}%
\put(135.1506,-170.7165){\makebox(0,0)[rb]{$B_2$}}%
\end{picture}%

\caption{Symplectic basis}
\label{cycles}
\end{figure}
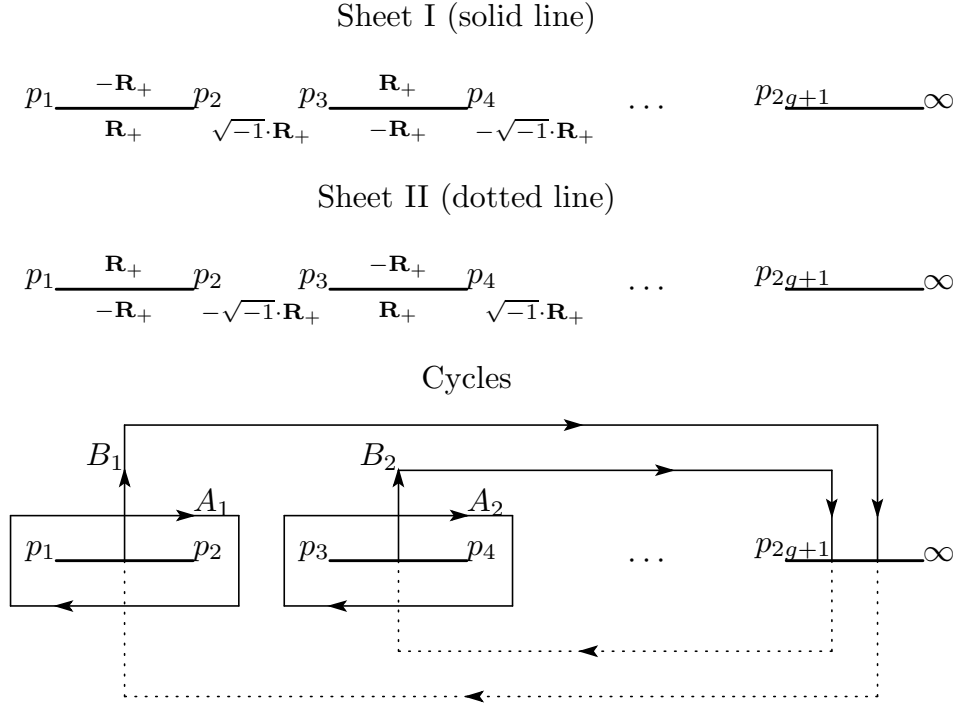
Note that the cycles satisfy 
$$
A_i\cdot A_j=
B_i\cdot B_j=0,\quad
A_i\cdot B_j=\delta_{ij}
$$ 
for $1\leq i,j \leq g$
under the intersection form.

We define holomorphic forms $\omega_j$ for $j=1, \dots,g$ as
$$
\omega_j=\frac{x^{j-1}dx}{y}.
$$
We define integrals $T_{i}^{(j)}$
by
$$
T_{i}^{(j)}=
\int_{p_i}^{p_{i+1}}\frac{x^{j-1}dx}
{\sqrt{\prod_{k=1}^i(x-p_k)\prod_{k=i+1}^{2g+1}(p_k-x)}}
$$
for $1\le i\le 2g$ and $1\le j\le g.$ 
Then the integrals $T_{i}^{(j)}$ are positive real numbers.
Using these integrals, we express the period integrals of $C$:
\begin{align*}
\int_{A_i}\omega_j=(-1)^i2T_{2i-1}^{(j)},\quad
\int_{B_i}\omega_j=2\sqrt{-1}(\sum_{k=i}^g(-1)^{k+1}T_{2k}^{(j)}).
\end{align*}
We set
\begin{equation}
\label{unnormalized period matirx}
A=(\int_{A_i}\omega_j )_{ij},\quad
B=
(\int_{B_i}\omega_j )_{ij}
\end{equation}
and consider the normalized period matrix $\tau$ by A-period:
\begin{equation}
\label{definition of normalized period matrix}
\tau=BA^{-1}.
\end{equation}
By Riemann's bilinear relations, $\det(A)$ is a non-zero real number 
and $\tau$ is a symmetric matrix whose imaginary part is positive definite.
Note also that $\tau$ is purely imaginary. 
\begin{remark}
Since the Vandermonde matrix
$\det(x_i^{j-1})_{1\leq i,j\leq g}$ is positive on
$p_{2i-1}\leq x_i\leq p_{2i}$, $(-1)^{g(g+1)/2}\det(A)$ is positive.
\end{remark}
For $I=(i_1, \dots, i_g)\in \mathbf F_2^g$, we define theta constants as
\begin{align*}
\theta_{I}(\tau)=
\sum_{n\in \mathbf{Z}^g}\exp (\pi \sqrt{-1}\cdot n \tau\; ^t n+ \pi 
\sqrt{-1}\cdot n\cdot ^t\!I).
\end{align*}

\begin{proposition}
\label{positivity of theta}
Let $M$ be a positive definite symmetric
$g\times g$ real matrix.
Then $\theta_I(\sqrt{-1}M)$
is positive for each $I\in \mathbf{F}_2^g$.
\end{proposition}
\begin{proof}
By the inversion formula of the theta function in \cite{M1} p.195, we have
$$
\sqrt{\det(M)}\cdot\theta_I(\sqrt{-1}M)=
\sum_{n\in\mathbf Z^g}\exp\left(\sqrt{-1}\pi (n+\frac{I}{2})
(\sqrt{-1}M^{-1})\;^t(n+\frac{I}{2})\right),
$$
where $\sqrt{\det(M)}$ takes a positive value.
Since each term of the right hand side is positive, 
the left hand side is positive.
\end{proof}

We consider variable $u=(u_I)_{I\in \mathbf F_2^g}$ whose coordinates
are indexed by $\mathbf F_2^g$. The pair $(\theta_I(\tau))_{I}$
is denoted by $\theta(\tau)$.
For $I\in \mathbf F_2^g$ , we define quadratic polynomials 
$F_I(u)$
of $2^g$ variables $u=(u_I)_{I\in \mathbf F_2^{g}}$ by
$$
F_{I}(u)=\frac{1}{2^g}\sum_{P\in \mathbf F_2^g}u_{I+P}u_{P}.
$$
We remark that the coefficients of $2^{g}F_I(u)$ 
are in $\mathbf Z_{\geq 0}$.
By Theorem 2 in \cite{I} p.139, we have
$2\tau$-formulas of theta constants  
\begin{align}
\label{twice formular}
&\theta_{I}(2\tau)^2=
F_{I}(\theta(\tau))
\end{align}
for $I\in \mathbf{F}_2^g$.

Now prepare some combinatorial notations 
for the statement of Thomae's formula. 
For an index $I\in \mathbf F_2^g$,
we define a subset $S_I$ of $R=\{1, \dots,2g+1,\infty\}$ as follows.
Let $\eta_i$ 
be elements of $M(2,g,\mathbf F_2)$ defined as
\begin{align*}
\eta_{2i-1} & =
\left(
\begin{matrix}0&\cdots& 0&\overset{\text{$i$-th}}{1}
& 0 & \cdots & 0 \\
1&\cdots& 1 & 0& 0 & \cdots & 0 
\end{matrix}
\right), \\
\eta_{2i} & =
\left(
\begin{matrix}0&\cdots& 0&\overset{\text{$i$-th}}{1}
& 0 & \cdots & 0 \\
1&\cdots& 1 & 1& 0 & \cdots & 0 
\end{matrix}
\right),
\end{align*}
for $i=1, \dots, 2g+1$. 
Then a subset $T_I$ of $R-\{2g+1,\infty\}=\{1,2,\dots, 2g\}$
is characterized by the equality
$$
\left(
\begin{matrix}0 \\
I
\end{matrix}
\right)=\sum_{j\in T_I}\eta_j.
$$
We set 
$$
S_I=\left\{
\begin{matrix}
T_I &\textrm{if }\# T_I \textrm{ is even,}\\
T_I\cup \{2g+1\} &\textrm{if }\# T_I \textrm{ is odd. }
\end{matrix}
\right.
$$
Let $U$ be the set $\{1,3,5,\dots, 2g+1\}$ and 
$R_1\circ R_2$  be the symmetric difference
of sets $R_1$ and $R_2$.
\begin{proposition}[\cite{M2} p.120, \cite{F}]
Let $A$ be the period matrix of $C$ in 
(\ref{unnormalized period matirx}).
Then we have
\begin{align}
\label{Thomae formula}
&\frac{(2\pi )^{2g}\theta_{I}(\tau)^4}{\det (A)^2}=
\prod_{i<j,i,j\in S_I\circ U}
(p_j-p_i)
\prod_{i<j,i,j\notin S_I\circ U}
(p_j-p_i).
\end{align}
\end{proposition}

\noindent
Here we used the fact that $\theta_I(\tau)$ is a real number
to determine the sign of Thomae's formula in \cite{M2}.

\section{Statement and proof of the main theorem}

\begin{definition}
[AGM sequences]
\label{def of recursive definition}

\par\noindent
\begin{enumerate}
\item
For an element $u=(u_I)_I\in \mathbf R_+^{2^g}$, we define
the termwise root $\sqrt{u}$ of $u$ by $(\sqrt{u_{I}})_I$.
\item
Let  $a=(a_I)_I$ be an element in $(\mathbf R_+)^{2^g}$.
We define $a_k=(a_{k,I})_I$ inductively by the relation
\begin{align*}
a_{0,I}=a_I,\quad a_{k+1,I}=
F_{I}(\sqrt{a_k}).
\end{align*}
\end{enumerate}
\end{definition}
A proof of the following proposition will be left to readers.
\begin{proposition-definition}[Generalized arithmetic-geometric mean]

For an element $a=(a_I)_I$ in $(\mathbf R_+)^{2^g}$,
the limits $\displaystyle\lim_{k\to \infty}a_{k,I}$ exist 
and are independent of
indexes $I$. This common limit is called the 
generalized arithmetic-geometric mean of $a=(a_I)_I$ 
and denoted by $\mu_g(a)$.
\end{proposition-definition}
\begin{problem}
Is it possible to express the generalized arithmetic-geometric mean 
$\mu_g(a)$ of $a=(a_I)_I\in (\mathbf R_+)^{2^g}$ in terms of period 
integrals of a family of varieties parametrized by $a$ ?
\end{problem}
\begin{theorem}
\label{main theorem}
Let $p_1<\cdots <p_{2g+1}$ be real numbers.
We define $a_{I}$ by
\begin{equation}
\label{initial}
a_{I}=\sqrt{
\prod_{i<j,i,j\in S_I\circ U}
(p_j-p_i)
\prod_{i<j,i,j\notin S_I\circ U}
(p_j-p_i)}.
\end{equation}
Then we have
$$
\mu_g(a)=\frac{(2\pi)^g}{\mid\det(A)\mid},
$$
where $A$ is the period matrix of $C$ in 
(\ref{unnormalized period matirx}).
\end{theorem}
\begin{proof}
By the initial condition, we have
$$
a_{0,I}= \frac{(2\pi)^g \theta_{I}(\tau)^2}{|\det(A)|}.
$$
We show that 
$$
a_{n,I}=
\frac{(2\pi)^g\theta_{I}(2^n\tau)^2}{|\det(A)|},
$$
by induction on $n$.
Since $\theta_I(2^n\tau)$ is a positive real number 
by Proposition \ref{positivity of theta}
for each $I$, 
we have
\begin{align*}
a_{n+1,I}=&
F(\sqrt{a_{n}}) \\
=&\frac{(2\pi)^g\cdot
F(\theta(2^n\tau)) }{|\det(A)|}
\qquad\text{(by the induction hypothesis)}\\
=&\frac{(2\pi)^g\cdot
\theta_I(2^{n+1}\tau)^2}{|\det(A)|} 
\qquad\text{(by the formula (\ref{twice formular}))}
\end{align*}
Therefore we have
$$
\lim_{n\to \infty}a_{n,I}=\frac{(2\pi)^g}{|\det(A)|}.
$$
\end{proof}
\section{Period of Calabi-Yau variety of certain type}
\label{Calabi-Yau period}
We study a relation between the generalized arithmetic-geometric mean of
the last section and a period of a Gorenstein Calabi-Yau 
variety of a certain type.

\begin{definition}[Calabi-Yau varieties]
A variety $X$ only with Gorenstein singularities is called 
a Calabi-Yau variety if the dualizing sheaf of $X$ is trivial and
$X$ has a global crepant resolution.
\end{definition}

Let $\mathbf P=\mathbf P^{g}$ be the $g$ dimensional projective space 
and $H_1\cdots H_{2g+2}$
be hyperplanes of $\mathbf P$. There is a unique line bundle $\Cal L$
on $\mathbf P$ and a unique isomorphism
$\varphi:\Cal L^{\otimes 2}\simeq O_X(-\sum_{i=1}^{2g+2}H_i)$ up to 
a non-zero constant. Using the isomorphism $\varphi$, we define a 
double covering $X={\Cal Spec}(\Cal O_X\oplus \Cal L)$, where the
multiplication on $\Cal L\otimes \Cal L\to \Cal O_X$ is given by the
isomorphism $\varphi$.

By the following Proposition 
\ref{existence of global crepant resolution}, 
$X$ becomes a Calabi-Yau variety,
since it admits a global crepant resolution.
\begin{proposition}
\label{existence of global crepant resolution} $ $
\begin{enumerate}
\item
If $\cup_{i=1}^{2g+2}H_i$ is normal crossing, then the variety $X$
has only Gorenstein singularities. Also it admits a global
crepant resolution.
\item
Under the above hypotheses, the dualizing sheaf is isomorphic
to the structure sheaf.
\end{enumerate}
\end{proposition}
\begin{proof}
(1) 
We choose a local coordinate $\xi_1, \dots, \xi_g$ of $\bold P$
such that the 
divisor $\cup_{i=1}^{2g+2}H_i$ is locally defined by
$\xi_1\cdots \xi_h=0$ ($h\leq g$).
Then the variety $X$
is locally defined by the equation $\eta^2=\xi_1\cdots \xi_h$.
Therefore X is locally isomorphic to
$Spec(\sigma\check{}\cap M^*)\times \bold A^{g-h}$
, where
$$
M^*=\mathbf Z^h+(\frac{1}{2}, \dots ,
\frac{1}{2})\mathbf Z\subset \mathbf Q^h,
\quad \sigma\check{}=(\mathbf R_+)^h.
$$
Let $\sigma$ be
the dual simplex of $\sigma\check{}$ and $M$ be the dual lattice of $M^*$.
Since $\sigma$ is generated by elements contained primitive
hyperplanes, $X$ is Gorenstein.
We can construct a global crepant resolution as follows.
We make a refinement of the simplex $\sigma$ into a regular 
fan $\cup_{\mathbf w\in \rho_h}\sigma_{\mathbf w}$
indexed by the set $\rho_h$ of ``unfair tournament'' of 
$\{1,\dots, h\}$.  
A sequence $\mathbf w=(w_1,\dots, w_{h-1})$
is an element of the set $\rho_h$ 
if it satisfies the following properties:
\begin{enumerate}
\item[(i)]
$w_1$ is equal to $1$ or $2$ and 
\item[(ii)]
$w_{i}$ is equal to $w_{i-1}$ or $i+1$ for $2\leq i\leq h-1$.
\end{enumerate}
For an element $\mathbf w$ of $\rho_h$, 
we define $\sigma_{\mathbf w}$
as a cone generated by 
\begin{align*}
B_{\mathbf w}=\{&u_1=e_1+e_2,u_2=e_{w_1}+e_{3},
u_3=e_{w_2}+e_4 \dots, 
u_{h-1}=e_{w_{h-2}}+e_{h},\\
&u_h=2e_{w_{h-1}}\},
\end{align*}
where $e_i$ is the standard basis of
$\mathbf Z^h\supset M$. 
Since the set $B_{\mathbf w}$ is a free base of $M$, 
the fan $\cup_{\mathbf w\in \rho_h}\sigma_{\mathbf w}$ 
is regular and 
it defines a smooth 
toric variety $\tilde X$.
The coordinates associated to $\mathbf Z^h\subset M^*$ 
are written as
$\xi_1,\dots,\xi_h$. 
($\eta$ corresponds to $\frac{1}{2}(1, \dots, 1)$.)
Let $z_1, \dots ,z_h$ be the coordinates associated to
the dual base $B_{\mathbf w}$ of $M$. Then we have
$$
z_1^{u_1}\cdots z_h^{u_h}=\xi_1^{e_1}\cdots \xi_h^{e_h}.
$$
Thus $\xi_1^{\frac{1}{2}}\cdots \xi_h^{\frac{1}{2}}=
z_1\cdots z_h$. Let $\omega_X$ be the rational differential form
on $X$ defined by
$$
\omega_X=\xi_1^{-\frac{1}{2}}\cdots \xi_h^{-\frac{1}{2}}
d\xi_1\wedge \cdots d\xi_h\wedge d\xi_{h+1}\wedge\cdots \wedge
d\xi_g.
$$ 
It is a generator 
of the dualizing sheaf of $X$.
The pull back of $\omega_X$
to the affine toric variety associated to
$\sigma_{\mathbf w}$ is a non-zero constant multiple of
$$
dz_1\wedge\cdots \wedge dz_h\wedge d\xi_{h+1}\wedge\cdots \wedge
d\xi_g.$$
Therefore the 
map $\tilde X\times \bold A^{g-h}\to X$ is a crepant resolution.
Since the local crepant resolutions depend only on the choice of
order of the components of the branching divisor, they are patched
together into a global crepant resolution. 

\medskip 
\noindent
(2) 
In this proof, we use symbols $\xi_1,\dots, \xi_g$ as the global
inhomogeneous coordinates of $\mathbf P$ for the infinite hyperplane
$H_{2g+2}$. Let
$l_i=l_i(\xi)$ 
be an inhomogeneous linear form defining the hyperplane $H_i$ 
for $i=1, \dots, 2g+1$.
Then a defining equation of the double covering $X$ 
can be written as
$$
\eta^2=\prod_{i=1}^{2g+1}l_i(\xi).
$$
As is shown in the proof of (1),
\begin{equation}
\label{global holomorphic}
\omega_X= 
\frac{1}{\eta}d\xi_1\wedge \cdots \wedge d\xi_g
\end{equation}
is a global generator of the dualizing sheaf of $X$.
\end{proof}

For real numbers $p_1<\dots< p_{2g+1}$, we define linear forms $l_i$ by
$$
l_i=\xi_1-p_i\xi_2+p_i^2\xi_3+\cdots +(-1)^{g-1}p_i^{g-1}\xi_g+(-1)^gp_i^g
$$
and set $H_i=\{l_i=0\}$.
By using the Vandermonde matrix, 
we see that $\cup _{i=1}^{2g+2}H_i$ is a normal
crossing divisor. 

We define a subset $\Delta$ of $\mathbf R^g$ as
\begin{align*}
\Delta=\{(x_1, \dots, x_g)\mid &
(-1)^{i-1}l_{2i-1}(x_1, \dots, x_g) \geq 0 \text{ for }i=1, \dots, g+1, 
\text{ and }\\
&(-1)^{i}l_{2i}(x_1, \dots, x_g) \geq 0 \text{ for }i=1, \dots, g \}.
\end{align*}
We set
$$
\omega_X=\frac{1}{\eta}d \xi_1\wedge\cdots \wedge d \xi_g,\text{ and }
\gamma_\pm=\{(\xi,\eta)\in X \mid \xi\in \Delta, \pm \eta\geq 0 \}.
$$
Then $\gamma=\gamma_+-\gamma_-$ defines a $g$-chain in $X$.
We have the following relation between the generalized
arithmetic-geometric mean and a period of the Calabi-Yau variety $X$.
The following theorem is obtained by  Theorem 2 in \cite{T}. 
\begin{theorem} 
Let $a=(a_{I})_I$ be an element of $\mathbf R_+^g$
defined in (\ref{initial}). 
Under the above notation, we have
$$
\mu_g(a)=
\frac{2\pi^g}
{\int_{\gamma}\omega_X}.
$$
\end{theorem}
\begin{proof}
Let $C_j$ be a copy of the curve $C$ given by 
$y_j^2=\prod_{i=1}^{2g-1}(x_j-p_i)$.
We define a map $\pi:C_1\times \cdots \times C_g\to X$ by
sending $((x_1,y_1), \dots, (x_g,y_g))$ to the point
whose $\xi_k$-coordinate and $\eta$-coordinate are 
the $(g+1-k)$-th elementary symmetric
function of $x_1, \dots, x_g$ and $\prod_{j=1}^gy_j$, respectively.
Then we have 
$$
\pi^*\omega_X=\sum_{\sigma\in \Cal S_g}
\sgn (\sigma)\boxtimes_{i=1}^{g}\omega_{\sigma(i)}.
$$
Since
$
\pi_*(A_1\times \dots\times A_g)=(-1)^{g(g+1)/2}2^{g-1}\gamma,
$
we have
$$
2^{g-1}\int_{\gamma}\omega_X=\mid\det(A)\mid.
$$
By Theorem \ref{main theorem}, we have the theorem.
\end{proof}
\section{Genus two case}

In this section, we will give a detailed study for the case of $g=2$. 
Refer to \cite{B} and \cite{Me} for the original results by Borchardt
and recent related works by Mestre, respectively.
We begin with $(a_{00},a_{01},a_{10},a_{11})$ as initial data for 
AGM sequences.
The recursive relations for $a_{k,I}$ $(I\in \mathbf F_2^2, k=0,1,\cdots)$
are given as 
$a_{0,I}=a_I$ and
$a_{k+1,I}=F_I(\sqrt{a_{k,00}},\cdots,\sqrt{a_{k,11}})$, where
\begin{align*}
F_{00}(u_{00},u_{01},u_{10},u_{11})
&=\frac{1}{4}(u_{00}^2+u_{01}^2+u_{10}^2+u_{11}^2), \\
F_{01}(u_{00},u_{01},u_{10},u_{11})
&=\frac{1}{2}(u_{00}u_{01}+u_{11}u_{10}), \\
F_{10}(u_{00},u_{01},u_{10},u_{11})
&=\frac{1}{2}(u_{00}u_{10}+u_{11}u_{01}), \\
F_{11}(u_{00},u_{01},u_{10},u_{11})
&=\frac{1}{2}(u_{00}u_{11}+u_{10}u_{01}). \\
\end{align*}

In the following,
we assume that $a_{00}>a_{10}>a_{11}>a_{01}$ and
$a_{00}a_{01}>a_{10}a_{11}$. 
First we define positive real numbers $k_1> k_2$ and 
$0<l_2<l_1<1$ such that
$$
(a_{00}+a_{01})^2-(a_{10}+a_{11})^2=k_1^2,\quad
(a_{00}-a_{01})^2-(a_{10}-a_{11})^2=k_2^2,
$$
\begin{align*}
a_{00}+a_{01}=\frac{1+l_1^2}{1-l_1^2}k_1,\quad
a_{10}+a_{11}=\frac{2l_1}{1-l_1^2}k_1, \\
a_{00}-a_{01}=\frac{1+l_2^2}{1-l_2^2}k_2,\quad
a_{10}-a_{11}=\frac{2l_2}{1-l_2^2}k_2,
\end{align*}
We set
\begin{align*}
&p_1=0,\quad
p_2=\displaystyle \frac {1}{(1 - {l_2}^{2})\,(1
 - {l_1}^{2})},
\\
&
{p_3} = {\displaystyle \frac {2({l_1}\,{l_2} + 
1)\, a_{00}
}{(1 - {l_1}^{2})\,
(1 - {l_2}^{2})\,({k_1} + {k_2})\,(1 - {
l_1}\,{l_2})}},
\\
&
{p_4} = {\displaystyle \frac {2({l_1}\,{l_2} + 
1)\,a_{01}
}{(1 - {l_1}^{2})\,
(1 - {l_2}^{2})\,({k_1} - {k_2})\,(1 - {
l_1}\,{l_2})}},
\\
&
{p_5} = {\displaystyle \frac {
4a_{00}a_{01}
}{(
{k_1} - {k_2})\,({k_1} + {k_2})\,(1 - 
{l_2}^{2})\,(1 - {l_1}^{2})}}. 
\end{align*}
Then we have
\begin{align}
\label{ratio relation}
(a_{00}^2:a_{01}^2:a_{10}^2:a_{11}^2)=
(&
(p_3-p_1)(p_5-p_1)(p_5-p_3)(p_4-p_2): \\
\nonumber &(p_4-p_1)(p_5-p_1)(p_5-p_4)(p_3-p_2): \\
\nonumber &(p_3-p_2)(p_5-p_2)(p_5-p_3)(p_4-p_1): \\
\nonumber &(p_4-p_2)(p_5-p_2)(p_5-p_4)(p_3-p_1)
).
\end{align}
Therefore by Theorem \ref{main theorem}, we have
\begin{align*}
&\lim_{n\to \infty} a_{n,00}=
\frac{4\pi ^{2}a_{00}}
{\mid \det(A)\mid \sqrt{
(p_3-p_1)(p_5-p_1)(p_5-p_3)(p_4-p_2)}}\\
&=\frac{8\pi ^{2}}{\mid \det(A)\mid }\cdot
(1-l_1^2)^2(1-l_2^2)^2
\sqrt{\frac{(a_{00}a_{01}-a_{10}a_{11})^3
(1-l_1l_2)^3}
{a_{00}a_{01}a_{10}a_{11}(l_1^2-l_2^2)(1+l_1l_2)}}.
\end{align*}
where $A$ is the period matrix of 
$C$ in (\ref{unnormalized period matirx}).

Using the result of \S \ref{Calabi-Yau period}, we have
$$
\mid \det (A)\mid =
4\cdot \int_{\Delta}\frac{d\xi_1\wedge d\xi_2}
{\sqrt{\prod_{i=1}^5(\xi_1-p_i\xi_2+p_i^2)}},
$$
where $\Delta$ is a domain in $\mathbf R^2$ defined by
$l_1 \geq 0, -l_2\geq 0,-l_3\geq 0,
l_4\geq 0$ and $l_5\geq 0$. This is a period integral of the
covering $X$ of $\mathbf P^2$ defined by
$$
\eta^2=\prod_{i=1}^5(\xi_1-p_i\xi_2+p_i^2).
$$
We notice that the variety $X$ is the (nodal) Kummer 
surface of the Jacobian of $C$.
\begin{remark}
When $$a_{00}>a_{01}=a_{10}>a_{11},\quad a_{00}a_{11}>a_{01}a_{10},$$
$\mu_2(a)$ can be expressed 
in terms of the arithmetic-geometric mean $\mu_1$ and expressions 
$p_2,\dots,p_5$ by $a$ (see \cite{Me}).
\end{remark}

\bigskip
\begin{flushleft}
\begin{minipage}{7.0cm}
Keiji \textsc{Matsumoto}\\
Department of Mathematics\\
Hokkaido University\\
Sapporo, 060-0810, Japan\\
e-mail: matsu@math.sci.hokudai.ac.jp\\
\end{minipage}
\end{flushleft}

\begin{flushleft}
\begin{minipage}{7.0cm}
Tomohide \textsc{Terasoma}\\
Graduate School of Mathematical Sciences\\
The University of Tokyo \\
Komaba, Meguro, Tokyo, 153-8914, Japan\\
e-mail: terasoma@ms.u-tokyo.ac.jp\\
\end{minipage}
\end{flushleft}

\end{document}